\documentclass{amsart}

\usepackage[utf8]{inputenc}
\usepackage[english]{babel}
\usepackage{amsthm}
\usepackage{amssymb}
\usepackage{pgf,tikz}

\numberwithin{equation}{section}

\newtheorem{proposition}{Proposition}[section]
\newtheorem{lemma}[proposition]{Lemma}
\newtheorem{theorem}[proposition]{Theorem}
\newtheorem{corollary}[proposition]{Corollary}
\theoremstyle{definition}
\newtheorem{remark}[proposition]{Remark}

\newtheorem{example}[proposition]{Example}

\title{Equilateral $p$-gons in $\mathbb R^d$ and  deformed spheres and mod $p$ Fadell-Husseini index}
\author{Andrés Angel, Jerson Borja$^1$}

\date{\today}
\thanks{$^1$  The first author acknowledges and thanks the finacial support of the  grant \textit{ P12.160422.004/01- FAPA ANDRES ANGEL} from Vicedecanatura de Investigaciones de la Facultad de Ciencias de la Universidad de
los Andes, Colombia. The second author acknowledges and thanks the financial support of the grant \textit{Proyecto Semilla 2016-2} from Vicedecanatura de Investigaciones de la Facultad de Ciencias de la Universidad de
los Andes, Colombia}

\begin{document}
\maketitle
\begin{abstract}
We introduce the concept of $r$-equilateral $m$-gons. We prove the existence of $r$-equilateral $p$-gons in $\mathbb R^d$ if $r<d$ and the existence of equilateral $p$-gons in the image of continuous injective maps $f:S^d\to \mathbb R^{d+1}$. Our ideas are based mainly in the paper of Y. Soibelman \cite{soibelman}, in which the topological Borsuk number of $\mathbb{R}^2$ is calculated by means of topological methods and the paper of P. Blagojevi\'c and G. Ziegler \cite{blagojevictetrahedra} where Fadell-Husseini index is used for solving a problem related to the topological Borsuk problem for $\mathbb{R}^3$.

\end{abstract}

\section{Introduction}

Let $(X,\rho)$ be a metric space. For a compact subset $K$ of $X$, $b_{(X,\rho)}(K)$ denotes the \textit{Borsuk number} of $K$, that is, the minimal number of parts of $K$ of smaller diameter necessary to partition $K$. The \textit{Borsuk number} of $(X,\rho)$ is defined to be $B(X,\rho)=\max _{K}b_{X,\rho}(K)$. If $\Omega(\rho)$ denotes the set of metrics on $X$ equivalent to $\rho$, then the \textit{topological Borsuk number} of $(X,\rho)$ is defined by $B(X)=\min_{\tau\in\Omega(\rho)}B(X,\tau)$.

The \textit{topological Borsuk problem} for $\mathbb R^d$ is to estimate the \textit{topological Borsuk number} of $\mathbb R^d$, $B(\mathbb{R}^n)$ for the Euclidean metric (see \cite{soibelman}); one wants to know if $B(\mathbb{R}^n)\geq n+1$. Y. Soibelman proved in \cite{soibelman} that $B(\mathbb{R}^2)=3$, but it is not known if $B(\mathbb{R}^3)\geq 4$.

A nonempty subset $S$ of a metric space $(X,\rho)$ is called \textit{equilateral (with respect to $\rho$)}, if there is some positive real number $c$ such that $\rho(x,y)=c$ for all $x,y\in S$, $x\neq y$. In order to prove that $B(\mathbb R^d)\geq m$, it suffices to show that for any metric $\rho$ defined on $\mathbb R^d$ equivalent to the Euclidean metric, there exists some equilateral subset $S\subseteq (\mathbb R^d,\rho)$ of size $m$ (see \cite{soibelman}). The following result is due to C. M Petty (\cite{petty}):

\begin{theorem}\label{petty} Every equilateral set of size 3 in a finite dimensional normed space of dimension at least 3, can be extended to an equilateral set of size 4.
\end{theorem}

Theorem \ref{petty} implies that for any metric $\rho$ on $\mathbb R^3$ induced by a norm we have $B(\mathbb R^2, \rho)\geq 4$. In fact, restrict $\rho$ to a 2-plane $X$ in $\mathbb R^3$ and apply the result of Soibelman that the Borsuk number of $\mathbb R^2$ is 3 to obtain a $\rho$-equilateral set in $X$ of size 3. This equilateral set in $X$ is also an equilateral set of size 3 in $(\mathbb R^3,\rho)$, so by theorem \ref{petty}, it extends to a $\rho$-equilateral set of size 4. Thus, $B(\mathbb R^4, \rho)\geq 4$. This gives us a partial answer to the Borsuk problem for $\mathbb R^3$.

An \textit{$m$-gon} in a metric space $(X,\rho)$ is an $m$-cycle graph $G$ whose vertices are $m$ distinct elements in $X$. If $G$ has vertices $x_0,x_1,\ldots,x_{m-1}$ and edges $\{x_i,x_{i+1}\}$ for $i=0,1,\ldots,m-1$ (where $x_{m}=x_0$), we say that $G$ is an \textit{equilateral $m$-gon} in $(X,\rho)$, or a \textit{equilateral $m$-gon with respect to $\rho$}, if there is some constant $c>0$ such that $\rho(x_i,x_{i+1})=c$ for all $i=0,1,\ldots,m-1$. Given $r$ metrics on $X$, $\rho_1,\ldots,\rho_r$, we say that $G$ is an \textit{$r$-equilateral $m$-gon} (with respect to $\rho_1,\ldots,\rho_r$) if it is equilateral for each metric $\rho_i$.

Let $\mathbb{Z}/m=\langle \omega|\ \omega^m=1\rangle$. Given an $m$-gon $G$ with vertices $x_0,x_1,\ldots,x_{m-1}$ and edges $\{x_i,x_{i+1}\}$ for $i=0,1,\ldots,m-1$ (where $x_{m}=x_0$), $\mathbb{Z}/m$ acts naturally on the set of vertices of $G$ by $\omega\cdot x_i=x_{i+1}$, and this induces an action of $\mathbb{Z}/m$ on the set of edges of $G$: $\omega\cdot \{x_i,x_{i+1}\}=\{\omega\cdot x_{i},\omega\cdot x_{i+1}\}$.

If $K_m$ denotes the complete graph on vertices $x_0,\ldots,x_{m-1}$, then the action of $\mathbb{Z}/m$ on $x_0,x_1,\ldots, x_{m-1}$ induces an action of $\mathbb Z/m$ on the edges of $K_m$. The set of edges of $K_m$ decomposes into a union of disjoint orbits; some of these orbits represent $m$-gons, some other are $n$-gons for some $n<m$ and some other orbits are not even $n$-gons for any $n$ (see figure \ref{figmgons} (a)). If $m=p$ is an odd prime number, then each orbit of $\mathbb{Z}/p$ acting on the set of edges of $K_p$ is actually a $p$-gon, and there are $(p-1)/2$ of them (see figure \ref{figmgons} (b)). If the subscripts of the $x_i$'s are thought as the elements of the finite field $\mathbb{F}_p$, then these $(p-2)/2$ $p$-gons, that we will denote by $C_1,C_2,\ldots, C_{(p-1)/2}$, can be described as follows: $C_t$ has edges $\{x_{0}, x_{t}\}, \{x_{t}, x_{2t}\}, \ldots, \{x_{(p-1)t}, x_{0t}\}$. Our first main result is the following:

\begin{theorem}\label{equilateral}
Let $p$ be an odd prime number, $\rho_1,\rho_2,\ldots,\rho_r$ be $r$  metrics on $\mathbb{R}^d$ that are equivalent to the Euclidean metric of $\mathbb{R}^d$ and fix a sequence $\{t_1,t_2,\ldots,t_r\}$, where $t_i\in\{1,2,\ldots,(p-1)/2\}$. If $d>r$, then there are $p$ distinct points $x_0,x_1\ldots,x_{p-1}$ in $\mathbb{R}^d$ such each $p$-gon $C_{t_i}$ is $\rho_i$-regular, $i=1,2,\ldots,r$.
\end{theorem}

As a consequence, if we put $t_1=t_2=\cdots=t_r$ in theorem \ref{equilateral}, then we have that there exists an $r$-equilateral $p$-gon in $\mathbb R^d$ with respect to $\rho_1,\rho_2,\ldots,\rho_r$.

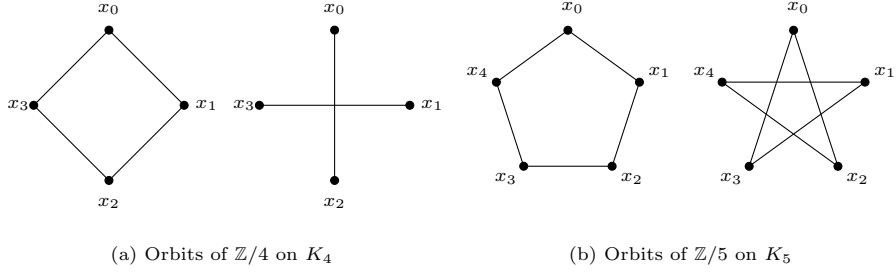
\begin{figure}
\begin{center}
\begin{tikzpicture}
\draw (3.0,4.0)-- (2.0,3.0);
\draw (2.0,3.0)-- (3.0,2.0);
\draw (3.0,2.0)-- (4.0,3.0);
\draw (4.0,3.0)-- (3.0,4.0);
\draw (6.0,4.0)-- (6.0,2.0);
\draw (5.0,3.0)-- (7.0,3.0);
\begin{scriptsize}
\draw [fill=black] (3.0,4.0) circle (1.5pt);
\draw[color=black] (3.,4.28) node {$x_0$};
\draw [fill=black] (2.0,3.0) circle (1.5pt);
\draw[color=black] (1.8,3) node {$x_3$};
\draw [fill=black] (3.0,2.0) circle (1.5pt);
\draw[color=black] (3,1.7) node {$x_2$};
\draw [fill=black] (4.0,3.0) circle (1.5pt);
\draw[color=black] (4.3,3) node {$x_1$};
\draw [fill=black] (6.0,4.0) circle (1.5pt);

\draw[color=black] (6,4.28) node {$x_0$};
\draw [fill=black] (6.0,2.0) circle (1.5pt);
\draw[color=black] (6,1.7) node {$x_2$};
\draw [fill=black] (5.0,3.0) circle (1.5pt);
\draw[color=black] (4.8,3) node {$x_3$};
\draw [fill=black] (7.0,3.0) circle (1.5pt);
\draw[color=black] (7.3,3) node {$x_1$};
\draw[color=black] (4.5, 1) node {(a) Orbits of $\mathbb{Z}/4$ on $K_4$};
\end{scriptsize}
\end{tikzpicture}
\begin{tikzpicture}
\draw (0.0,2.0)-- (-0.9510565162951535,1.3090169943749475);
\draw (-0.9510565162951535,1.3090169943749475)-- (-0.5877852522924732,1-0.8090169943749473);
\draw (-0.5877852522924732,1-0.8090169943749473)-- (0.5877852522924729,1-0.8090169943749476);
\draw (0.5877852522924729,1-0.8090169943749476)-- (0.9510565162951535,1.3090169943749474);
\draw (0.9510565162951535,1.3090169943749474)-- (0.0,2.0);

\draw (3.0,2.0)-- (2.41,0.18999999999999995);
\draw (2.41,0.18999999999999995)-- (3.95,1.31);
\draw (3.95,1.31)-- (2.05,1.31);
\draw (2.05,1.31)-- (3.59,0.18999999999999995);
\draw (3.59,0.18999999999999995)-- (3.0,2.0);

\begin{scriptsize}
\draw [fill=black] (0.95105651629515351,1.3090169943749474) circle (1.5pt);
\draw [fill=black] (-0.9510565162951535,1.3090169943749475) circle (1.5pt);
\draw[color=black] (-1.2,1.443348412702339) node {$x_4$};
\draw [fill=black] (-0.5877852522924732,1-0.8090169943749473) circle (1.5pt);
\draw[color=black] (-0.3-0.524019793796178,0.7-0.6762692998010198) node {$x_3$};
\draw [fill=black] (0.5877852522924729,1-0.8090169943749476) circle (1.5pt);
\draw[color=black] (0.8515788043323498,0.7-0.6762692998010198) node {$x_2$};
\draw [fill=black] (0.0,2.0) circle (1.5pt);
\draw[color=black] (0.06377950526808593,2.3) node {$x_0$};
\draw[color=black] (1.219762089446624,1.4416742063511672) node {$x_1$};

\draw [fill=black] (3.95,1.31) circle (1.5pt);
\draw[color=black] (4.219762089446624,1.4416742063511672) node {$x_1$};
\draw [fill=black] (2.05,1.31) circle (1.5pt);
\draw[color=black] (1.816411978190912,1.443348412702339) node {$x_4$};
\draw [fill=black] (2.41,0.18999999999999995) circle (1.5pt);
\draw[color=black] (2.1709575871503096,0.7-0.6762692998010198) node {$x_3$};
\draw [fill=black] (3.59,0.18999999999999995) circle (1.5pt);
\draw[color=black] (3.8515788043323498,0.7-0.6762692998010198) node {$x_2$};
\draw [fill=black] (3.0,2.0) circle (1.5pt);
\draw[color=black] (3.06377950526808593,2.3) node {$x_0$};

\draw[color=black] (1.5, -1) node {(b) Orbits of $\mathbb{Z}/5$ on $K_5$};
\end{scriptsize}
\end{tikzpicture}
\caption{Orbits of $\mathbb{Z}/m$ acting on $K_m$}
\label{figmgons}
\end{center}
\end{figure}

To state our second main result we first introduce some terminology. Let $p$ be an odd prime and $d\geq 1$. We call the pair $(d,p)$ \textit{admissible} if given any continuous injective map $f:S^{d}\to \mathbb{R}^{d+1}$ and $(p-1)/2$ metrics $\rho_1,\ldots,\rho_{(p-1)/2}$ on $\mathbb{R}^{d+1}$, each metric equivalent to the Euclidean metric of $\mathbb{R}^{d+1}$, then there can be found $p$ different points in the image of $f$, $f(x_0), \ldots,f(x_{p-1})$, such that
\begin{equation}\rho_j(f(x_0), f(x_{j})) =\rho_j(f(x_{j}), f(x_{2j}))=\cdots=\rho_j(f(x_{(p-1)j}), f(x_0)),\end{equation}
for all $j=1,2,\ldots,(p-1)/2$. In other words, the $(p-1)/2$ $p$-gons determined by $f(x_0), \ldots,f(x_{p-1})$, $C_1,C_2,\ldots, C_{(p-1)/2}$ satisfy that each $C_i$ is equilateral with respect to $\rho_i$, $i=1,2,\ldots, (p-1)/2$. Now we can state our result as follows:

\begin{theorem}\label{pgonsdeformed}
If $p\leq 2d+1$ and $(p-1)^2/2$ is not of the form $jd+1$ for any $j=1,2,\ldots,p-1$, then the pair $(d,p)$ is admissible.
\end{theorem}

Both results, Theorem \ref{equilateral} and Theorem \ref{pgonsdeformed} are proved by means of topological methods: we use the \textit{configuration space/test map scheme} and the \textit{Fadell-Husseini index with coefficients in $\mathbb F_p$} of the involved spaces to solve our problems.  

\section{Fadell-Husseini index}

Let $G$ be a finite group and $R$ a commutative ring with unit. The Fadell-Husseini index with coefficients in $R$ assigns to each $G$-space $X$ an ideal $Ind_{G,R}X$ of the ring $H^{\ast}(BG;R)$, the cohomology ring of the classifying space $BG$ with coefficients in $R$. The important property of the Fadell-Husseini index is that it gives a necessary condition for the existence of equivariant maps between to $G$-spaces: if $X$ and $Y$ are $G$-spaces and there is some $G$-equivariant map $f:X\to Y$, then $Ind_{G,R}Y\subseteq Ind_{G,R}X$. For definitions, basic properties and some calculations see \cite{blagojevicequivariant, blagojevictetrahedra, blagojevicdihedral, fadellindexvalued, fadellrelative, zivaljevic2}. Here, we want to mention the properties and calculations of the Fadell-Husseini index we will use in the proofs of our results. 

For a finite group $G$ we have the associated universal $G$-bundle $EG\to BG$; if $X$ is a $G$-space $X$, then there is the associated \textit{Borel construction} $X_G=X\times_GEG$ and a map $q_X:X_G\to BG$ induced by the unique map $X\to \ast$. The index $Ind_{G,R}X$ is defined as the kernel of the induced map in cohomology $q_X^{\ast}:H^{\ast}(BG;R)\to H^{\ast}(X_G;R)$. The \textit{axioms of an ideal valued index theory} are satisfied by Fadell-Husseini index:\\

\noindent\textbf{Monotonicity:} If there is some $G$-equivariant map $f:X\to Y$, for $G$-spaces $X$ and $Y$, then 
\[Ind_G Y\subseteq Ind_G X.\]
\textbf{Additivity:} If $\{X=X_1\cup X_2,X_1,X_2\}$ is excisive, where $X_1$ and $X_2$ are $G$-invariant subspaces of $X$, then 
\[Ind_G\ X_1\cdot Ind_G\ X_2\subseteq Ind\ X.\]
\textbf{Continuity:} If $A$ is a closed $G$-invariant subspace of $X$, then there is a $G$-invariant neighborhood $U$ of $A$ in $X$ such that 
\[Ind_G\ A=Ind_G\ U.\]
\textbf{Index theorem:} If $f:X\to Y$ is $G$-equivariant, $B\subseteq Y$ a closed $G$-invariant subspace and $A:=f^{-1}(B)$, then 
\[Ind_G\ A\cdot Ind_G\ (Y\setminus B)\subseteq Ind_G\ X.\]

For instance, if $X$ is a $G$-space and $Y$ an $H$-space, then $X\times Y$ is a $G\times H$-space; the map $q_{X\times Y}:(X\times Y)_{G\times H}\to B(G\times H)$ can be identified with the map $q_X\times q_Y:X_G\times Y_H\to BG\times BH$ and the map $q_{X\times Y}^{\ast}$ can be identified (under K\"unneth isomorphisms) with 
\[q_X^{\ast}\otimes q_Y^{\ast}: H^{\ast}(BG ;\mathbb{K})\otimes H^{\ast}(BH; \mathbb{K})\to H_G^{\ast}(X; \mathbb{K})\otimes H_H^{\ast}(Y; \mathbb{K}),\]
where $\mathbb{K}$ is a field.
The kernel of this map can be expressed as follows:
 
\[\ker q_X^{\ast}\otimes q_Y^{\ast}=\ker q_X^{\ast}\otimes H^{\ast}(BH; \mathbb{K})+H^{\ast}(BG; \mathbb{K})\otimes \ker q_Y^{\ast},\]
that is,
\[Ind_{G\times H,\mathbb{K}} X\times Y=Ind_{G,\mathbb{K}}\otimes H^{\ast}(BH; \mathbb{K})+H^{\ast}(BG; \mathbb{K})\otimes Ind_{H,\mathbb{K}}Y.\]

As particular cases, if $Y=pt$, then 
\[Ind_{G\times H,\mathbb{K}}X=Ind_{G,\mathbb{K}}X\otimes H^{\ast}(BH; \mathbb{K});\]

and if $H^{\ast}(BG; \mathbb{K})\cong \mathbb{K}[x_1,\ldots, x_n], H^{\ast}(BH; \mathbb{K})\cong \mathbb{K}[y_1,\ldots, y_m]$, $Ind_{G,\mathbb{K}}X=\langle f_1,\ldots, f_r\rangle \text{ and } Ind_{H,\mathbb{K}}Y=\langle g_1,\ldots, g_s\rangle,$
then 
\[Ind_{G\times H,\mathbb{K}}X\times Y=\langle f_1,\ldots,f_r,g_1,\ldots, g_s\rangle \subseteq \mathbb{K}[x_1,\ldots,x_n,y_1,\ldots, y_m].\]
 
These results on the index of products can be consulted in \cite{fadellindexvalued}.

\noindent\textbf{Spectral Sequence.} The map $q_X:X_G\to BG$ is a fibration with fiber $X$. There is a  spectral sequence $\{E_r^{\ast,\ast},d_r\}$ with $E_2^{p,q}=H^p(BG;\mathcal{H}^q(X;R))$ where $\mathcal{H}^q(X;R)$ is a system of local coefficients, this local coefficient system is determined by the action of $G=\pi_1(BG)$ on $H^{\ast}(X;R)$. Therefore, the $E_2$-page of the spectral sequence is interpreted as the cohomology of the group $G$ with coefficients in the $G$-module $H^{\ast}(X,R)$ (see for example \cite{adem, blagojevicequivariant}),
\begin{equation}
E_2^{p,q}=H^p(G;H^q(X;R)).
\end{equation}
The map $q_X^{\ast}:H^{\ast}(BG,R)\to H^{\ast}(X_G,R)$ can be represented as the composition (known as the \textit{edge homomorphism})
\begin{equation}
H^{\ast}(BG;R)\to E_2^{\ast,0}\to E_3^{\ast,0}\to E_4^{\ast,0}\to\cdots\to E_{\infty}^{\ast,0}\subseteq H^{\ast}(X_G;R).
\end{equation}\\
 
\noindent\textbf{Spheres, $E_nG$-spaces and the configuration space $F(\mathbb R^d,p)$.} A free $G$-space $X$ is an \textit{$E_nG$-space} if it is an $(n-1)$-connected, $n$-dimensional CW-complex equipped with a $G$-invariant CW-structure. $E_nG$-spaces exist; in fact, the join $G^{\ast(n+1)}$ of $n+1$ copies of the group $G$, where $G$ is regarded as a 0-dimensional simplicial complex, is an $E_nG$-space.

The following properties of the Fadell-Husseini index can be found in \cite{zivaljevic2}.

\begin{proposition}\label{indexzivaljevic} i) For the antipodal $\mathbb Z/2$-action on a sphere $S^n$, we have that   
\[ Ind_{\mathbb Z/2,\mathbb{F}_2}S^n=\langle t^{n+1}\rangle \subseteq \mathbb{F}_2[t].\]
In general, if $X$ is an $E_n\mathbb Z/2$-space, then
\[Ind_{\mathbb Z/2,\mathbb{F}_2}X=\langle t^{n+1}\rangle \subseteq H^{\ast}(B\mathbb Z/2; \mathbb{F}_2)=\mathbb{F}_2[t].\]

 ii) For an odd prime $p$, we have $H^{\ast}(B\mathbb Z/p,\mathbb{F}_p)=\mathbb{F}_p[a,b]/\langle a^2\rangle$, where $\deg(a)=1$ and $\deg(b)=2$. The unit sphere $S^{2n-1}\subseteq \mathbb C^n$ is a $\mathbb Z/p$-space where $\mathbb Z/p$ is seen as the subgroup of $S^1$ of $p$-th roots of 1. $S^1$ acts on $S^{2n-1}$ by complex multiplication and so does $\mathbb Z/p$. Then    
\[Ind_{\mathbb Z/p,\mathbb{F}_p}S^{2n-1}=\langle b^n\rangle \subseteq \mathbb{F}_p[a,b]/\langle a^2\rangle.\] 
If $X$ is an $E_{2n-1}\mathbb Z/p$, then 
\[Ind_{\mathbb Z/p,\mathbb{F}_p}X=\langle b^{n}\rangle \subseteq \mathbb{F}_p[a,b]/\langle a^2\rangle.\]

 iii) For an $E_{2n}\mathbb Z/p$-space $X$, we have that 
\[Ind_{\mathbb Z/p,\mathbb{F}_p}X=\langle ab^{n}\rangle \subseteq \mathbb{F}_p[a,b]/\langle a^2\rangle.\]

 iv) For a finite group $G$ and an $E_nG$-space $X$, we have
\[
Ind_{G,\mathbb{K}}X\subseteq \bigoplus_{d>n}H^{d}(BG;\mathbb{K}),
\]
that is, for every (homogeneous) element $x\in Ind_GX\subseteq H^{\ast}(BG;\mathbb{K})$, $deg(x)>n$.\\

 v) Let $U, V$ be two $G$-representations and $S(U), S(V)$ be the associated unit $G$-spheres. Assume that the vector bundles $U\to U_G\to BG$ and $V\to V_G\to BG$ are orientable over $\mathbb{K}$. If $Ind_{G,\mathbb{K}}S(U)=\langle f\rangle \subseteq H^{\ast}(BG;\mathbb{K})$ and $Ind_{G,\mathbb{K}}S(V)=\langle g\rangle\subseteq H^{\ast}(BG;\mathbb{K})$, then 
\[Ind_{G,\mathbb{K}}S(U\oplus V)=\langle f\cdot g\rangle\subseteq H^{ \ast}(BG;\mathbb{K}).\]

 vi) Let $V=\mathbb C$ be the 1-dimensional complex $(\mathbb Z/p)^k$-representation associated with the vector $(\alpha_1,\ldots, \alpha_n)\in \mathbb{F}_p^n$ (if $\omega_i$ is the generator of the $i^{th}$ copy of $\mathbb Z/p$ in $(\mathbb Z/p)^n$, then the vector $\alpha$ is characterized by the equality $\omega_i\cdot v=\zeta^{\alpha_i}v$ where $\zeta=\large{e}^{2\pi i/p}$, $v\in V$ and $i=1,\ldots,n$). Then 
\[
Ind_{(\mathbb Z/p)^n,\mathbb{F}_p}S(V)=\langle \alpha_1b_1+\cdots+\alpha_nb_n\rangle\subseteq \mathbb{F}_p[b_1,\ldots, b_n]\subseteq H^{\ast}(B(\mathbb Z/p)^n;\mathbb{F}_p).
\]
\end{proposition}

The \textit{configuration space} of $n$ distinct points in the topological space $X$ is the space 
\begin{equation}
F(X,n)=\{(x_1,x_2,\ldots,x_n)\in X^n: x_i\neq x_j\text{ for all} i\neq j\}.
\end{equation}
The symmetric group $S_n$ acts on $F(X,n)$ by permuting coordinates; the cyclic group $\mathbb Z/n$ seen as the subgroup of $S_n$ generated by the permutation $(12\cdots n)$, acts on $F(X,n)$ by restricting the action of $S_n$. The computation of $Ind_{\mathbb Z/p,\mathbb F_p}F(\mathbb R^d)$, where $p$ is an odd prime number and $d>1$ is made in \cite{blagojevicequivariant} (Theorem 6.1): 

\begin{equation}\label{indexconfiguration}
Ind_{\mathbb Z/p,\mathbb{F}_p}F(\mathbb R^d,p)=\langle ab^{\frac{(d-1)(p-1)}{2}}, b^{\frac{(d-1)(p-1)}{2}+1}\rangle\subseteq \mathbb{F}_p[a,b]/\langle a^2\rangle.
\end{equation}

\section{Existence of $r$-equilateral $p$-gons in $\mathbb R^d$}

Assume that $p$ is an odd prime number, $r\geq 1$, $d\geq 1$, $\rho_1, \rho_2,\ldots,\rho_r$ are $r$ metrics on $\mathbb{R}^d$, each $\rho_i$ equivalent to the Euclidean metric of $\mathbb{R}^d$, and consider a sequence $\{t_1,t_2,\ldots, t_r\}$, where $t_i\in\{1,2,\ldots,(p-1)/2\}$. The problem is to show that there are $p$ distinct points $x_0,x_1\ldots,x_{p-1}$ in $\mathbb{R}^d$ such that $C_{t_i}$ is $\rho_i$-regular for all $i=1,2,\ldots,r$. 

The configuration space for our problem is the configuration space of $p$ distinct points in $\mathbb{R}^d$, $F(\mathbb{R}^d,p)=\{(x_0,x_1,\ldots, x_{p-1})\in(\mathbb{R}^d)^p:x_i\neq x_j, i\neq j\}$. Consider $\mathbb{Z}/p=\langle\omega|\ \omega^p=1\rangle$ acting on $(\mathbb{R}^d)^p$ by 
\begin{equation}\omega\cdot (x_0, x_1,\ldots,x_{p-1})=(x_{1}, x_{2},\ldots,x_{p-1},x_{0}).\end{equation}
The group $\mathbb{Z}/p$ acts freely on $F(\mathbb{R}^d,p)$.

Next, we define $V=\mathbb{R}^p\times \cdots \times \mathbb{R}^p$, the product of $r$ copies of $\mathbb{R}^d$, and define our test map $f=(f_1,f_2,\ldots,f_{r}):F(\mathbb{R}^d,p)\to V$, where  
\begin{equation} f_i(x_0,\ldots,x_{p-1})=(\rho_i(x_0,x_{t_i}), \rho_i(x_{t_i},x_{2t_i}), \ldots, \rho_i(x_{(p-1)t_i},x_0)).\end{equation}
The group $\mathbb{Z}/p$ acts on $\mathbb{R}^p$ by $\omega\cdot(y_0, y_1,\ldots,y_{p-1})=(y_{1}, y_{2},\ldots,y_{p-1},y_{0})$ and $\mathbb{Z}/p$ acts on $V$ diagonally. The map $f:F(\mathbb{R}^d,p)\to V$ becomes $\mathbb{Z}/p$-equivariant.

The existence of $p$ distinct points $x_0,x_1\ldots,x_{p-1}$ in $\mathbb{R}^d$ such that $C_{t_i}$ is $\rho_i$-regular for any $i$ is equivalent to the existence of an element $(x_0,\ldots,x_{p-1})\in F(\mathbb{R}^d,p)$ such that for each $i=1,\ldots,r$, $f_i(x_0,\ldots,x_{p-1})\in \Delta$, where $\Delta=\{(t,t,\ldots,t)\in \mathbb{R}^d:t\in\mathbb{R}\}$; this is equivalent to say that $im\ f\cap \Delta^{r}\neq \varnothing$.

If we assume that $im\ f\cap\Delta^r=\varnothing$, then we get a $\mathbb{Z}/p$-equivariant map $f:F(\mathbb{R}^d,p)\to V\setminus \Delta^r$. Note that the action of $\mathbb{Z}/p$ on $\mathbb{R}^p\setminus \Delta$ is free and so is the action of $\mathbb{Z}/p$ on $V\setminus \Delta^r$. Thus, we have the following:

\begin{lemma}\label{lemmaequilateral}
If there is no $\mathbb Z/p$-equivariant map $F(\mathbb{R}^d,p) \to V\setminus \Delta^r$, then there are $p$ distinct points $x_0,x_1\ldots,x_{p-1}$ in $\mathbb{R}^d$ such that $C_{t_i}$ is $\rho_i$-regular for all $i=1,2,\ldots,r$.
\end{lemma}

\begin{lemma}\label{lemmaequilateral1}
Let $p$ be an odd prime number, $d\geq 1$ and $r\geq 1$. If $d>r$, then there is no $\mathbb{Z}/p$-equivariant map $F(\mathbb{R}^d,p) \to V\setminus \Delta^r$.
\end{lemma}

Lemma \ref{lemmaequilateral1}, combined with Lemma \ref{lemmaequilateral}, completes the proof of Theorem \ref{equilateral}.

\begin{proof}
Assume that there is a $\mathbb{Z}/p$-equivariant map $h:F(\mathbb{R}^d,p)\to V\setminus \Delta^r$. We also have an $\mathbb{Z}/p$-equivariant map $g:V\setminus \Delta^r\to S((\Delta^r)^{\perp})$ (projecting and normalizing). The composition $g\circ h$ gives a $\mathbb{Z}/p$-equivariant map $F(\mathbb{R}^d,p)\to S((\Delta^r)^{\perp})$. The actions of $\mathbb{Z}/p$ on $F(\mathbb{R}^d,p)$ and $S((\Delta^r)^{\perp})$ are free. 

Since $\Delta^r$ is an $r$-dimensional subspace of $V=(\mathbb{R}^p)^r$, we have that $(\Delta^r)^{\perp}$ is an $r(p-1)$-dimensional subspace of $V$. Thus, $S((\Delta^r)^{\perp})$ is a sphere of dimension $r(p-1)-1$.

We calculate Fadell-Husseini indexes. By Proposition \ref{indexzivaljevic}, $ii)$,

\[Ind_{\mathbb{Z}/p,\mathbb{F}_p}S((\Delta^r)^{\perp})=\langle b^{r(p-1)/2}\rangle\subseteq \mathbb{F}_p[a,b]/\langle a^2\rangle\] 
and by (\ref{indexconfiguration}),
\[Ind_{\mathbb{Z}/p,\mathbb{F}_p}F(\mathbb{R}^d,p)=\langle ab^{(d-1)(p-1)/2}, b^{(d-1)(p-1)/2+1}\rangle\subseteq \mathbb{F}_p[a,b]/\langle a^2\rangle.\]
The existence of a $\mathbb{Z}/p$-equivariant map $F(\mathbb{R}^d,p)\to S((\Delta^r)^{\perp})$ yields the inclusion $Ind_{\mathbb{Z}/p,\mathbb{F}_p}S((\Delta^r)^{\perp})\subseteq Ind_{\mathbb{Z}/p,\mathbb{F}_p}F(\mathbb{R}^d,p)$. From this we deduce that $r(p-1)>(d-1)(p-1)$, that is, $r>d-1$. This ends the proof.
\end{proof}

\begin{remark}
The continuity of the map $f=(f_1,f_2,\ldots,f_r)$ follows from the fact that a metric $\rho:\mathbb{R}^d\times \mathbb{R}^d\to \mathbb{R}$ that induces the usual topology of $\mathbb{R}^d$ is a symmetric and continuous function. Nothing else about the $r$ metrics is used in the proof of theorem \ref{equilateral}. So, in theorem \ref{equilateral} we can take $r$ symmetric continuous maps  $\rho_i:\mathbb{R}^d\times \mathbb{R}^d\to \mathbb{R}$; with this generality, the conclusion of theorem \ref{equilateral} is that for each $i\in\{1,2,\ldots,r\}$ there is some $c_i\in\mathbb{R}$ (not necessarily positive) such that 
\begin{equation}\rho_i(x_0,x_{t_i})= \rho_i(x_{t_i},x_{2t_i})= \ldots= \rho_i(x_{(p-1)t_i},x_0)=c_i.\end{equation}
\end{remark}

Now we list some consequences of Theorem \ref{equilateral}.

\begin{corollary}
Let $p$ be an odd prime number and $\rho_1,\rho_2,\ldots,\rho_{(p-1)/2}$ be $(p-1)/2$ metrics on $\mathbb{R}^d$ that are equivalent to the Euclidean metric of $\mathbb{R}^d$. If we have $d>(p-1)/2$, then there are $p$ distinct points $x_0,x_1,\ldots,x_{p-1}$ in $\mathbb{R}^d$ such that for each $t\in\{1,2,\ldots,(p-1)/2\}$, the $p$-gon $C_t$ is $\rho_t$-regular. 
\end{corollary}

\begin{proof}
Put $r=(p-1)/2$ and take the sequence $\{1,2,\ldots,(p-1)/2\}$ in Theorem \ref{equilateral}.
\end{proof}

\begin{corollary}
Let $p$ be an odd prime number and $\rho$ be a metric on $\mathbb{R}^d$ that is equivalent to the Euclidean metric of $\mathbb{R}^d$. If $d>(p-1)/2$, then there are $p$ distinct points $x_0,x_1,\ldots,x_{p-1}$ in $\mathbb{R}^d$ such that for each $t=1,2,\ldots,(p-1)/2$, the $p$-gon $C_t$ is regular with respect to $\rho$.
\end{corollary}

\begin{proof} In theorem \ref{equilateral} put $r=(p-1)/2$, $\rho_1=\rho_2=\cdots=\rho_{(p-1)/2}$ and take the sequence $\{1,2,\ldots,(p-1)/2\}$.
\end{proof}

\begin{corollary}\label{equilateralpgon}
Let $p$ be an odd prime number and $\rho_1,\rho_2,\ldots,\rho_r$ be $r$ metrics on $\mathbb{R}^d$ that are equivalent to the Euclidean metric of $\mathbb{R}^d$. If $d>r$, then there exists an $r$-equilateral $p$-gon in $\mathbb R^d$ with respect to $\rho_1,\rho_2,\ldots,\rho_r$. 
\end{corollary}

\begin{proof}
In theorem \ref{equilateral} take the constant sequence $\{1,1,\ldots,1\}$.
\end{proof}

\begin{corollary}
The Borsuk number of $\mathbb{R}^2$ is 3.
\end{corollary}

\begin{proof} In theorem \ref{equilateral} take $d=2$, $r=1$ and $p=3$. Note that an equilateral $3$-gon is the same as an equilateral set of size 3.
\end{proof}

\begin{remark}
Soibelman proves in \cite{soibelman} that the Borsuk number of $\mathbb{R}^2$ is 3 by the same method (the case $d=2, r=1,p=3$). He remarks that for the topological Borsuk problem in $\mathbb{R}^n$, $n>2$, there is an $S_{n+1}$-equivariant map $(\mathbb{R}^n)^{n+1}\to\mathbb{R}^{n(n+1)/2}$. As in the CS/TM scheme shown above, we get an equivariant map between spheres, but the actions of various subgroups of $S_{n+1}$ are not free on the target, so Dold's theorem cannot be applied. Our results are based on the idea of having actual free actions of some groups (i.e. $\mathbb{Z}/p$) on the unit sphere of some convenient spaces on the target. We do not get equilateral sets, but $r$-equilateral $p$-gons.
\end{remark}

\begin{example} \textit{Bi-equilateral triangles do not always exist in $\mathbb{R}^2$.} 
Theorem \ref{equilateral} or Corollary \ref{equilateralpgon} ensures that there exist \textit{bi-equilateral triangles} for any pair of  metrics $\rho_1$ and $\rho_2$ in $\mathbb{R}^d$ whenever $d>2$. We  will exhibit two metrics in $\mathbb{R}^2$ for which there are no bi-equilateral triangles.

Consider the Euclidean norm $\|(x,y) \|_2=\sqrt{x^2+y^2}$ and the norm $\| (x,y)\|_1=|x|+|y|$. By regarding $\mathbb{R}^2$ as the field of complex numbers $\mathbb{Z}$, for the Euclidean norm, we can assume, without lost of generality, that all equilateral triangles have their vertices living in the unit circle $\{z\in\mathbb{Z}: |z|=1\}$. Let $z_0,z_1,z_2$ be the cubic roots of 1: $z_0=1, z_2=\cos (2\pi/3)+i\sin (2\pi/3), z_2=\cos (4\pi/3)+i\sin (4\pi/3)$. Now we can describe all Euclidean equilateral triangles with their vertices in the unit circle. In fact, for $0\leq \theta\ \leq2\pi$, the 3 complex numbers 
$$
\begin{cases} z_0(\theta)=\cos \theta+i\sin \theta,\\
z_1(\theta)=\cos (\theta+2\pi/3)+i\sin (\theta+2\pi/3),\\ 
z_2(\theta)=\cos (\theta+4\pi/3)+i\sin (\theta+4\pi/3)
\end{cases}
$$
are the vertices of an equilateral triangle. It suffices to consider $0\leq \theta\leq 2\pi/3$ for obtaining all equilateral triangles with their vertices in the unit circle.\\
Now we calculate the length of the edges of such triangles using the taxicab norm:
$$
\begin{cases} d_{01}(\theta)&=\|z_1(\theta)-z_0(\theta)\|_1=|\cos (\theta +2\pi/3)-\cos \theta|+|\sin (\theta +2\pi/3)-\sin \theta|,\\
d_{02}(\theta)&=\|z_2(\theta)-z_0(\theta)\|_1=|\cos (\theta +4\pi/3)-\cos \theta|+|\sin (\theta +4\pi/3)-\sin \theta|,\\
d_{12}(\theta)&=\|z_2(\theta)-z_1(\theta)\|_1\\
&=|\cos (\theta +4\pi/3)-\cos (\theta+2\pi/3)|+|\sin (\theta +4\pi/3)-\sin (\theta+2\pi/3)|,
\end{cases}
$$
these expressions reduce to 
$$
\begin{cases} d_{01}(\theta)=\sqrt{3}|\sin (\theta +\pi/3)|+\sqrt{3}|\cos (\theta +\pi/3)|,\\
d_{02}(\theta)=\sqrt{3}|\sin (\theta +2\pi/3)|+\sqrt{3}|\cos (\theta +2\pi/3)|,\\
d_{12}(\theta)=\sqrt{3}|\sin (\theta +\pi)|+\sqrt{3}|\cos (\theta +\pi)|.
\end{cases}
$$
There is no $\theta$ between $0$ and $2\pi/3 $ such that $d_{01}(\theta)=d_{02}(\theta)=d_{12}(\theta)$. In fact, for $0\leq\theta\leq 2\pi/3$, the equality $d_{01}(\theta)=d_{02}(\theta)$ holds only for $\theta=0,\pi/4,\pi/2$. The common values $d_{01}(\theta)=d_{02}(\theta)$ are $\frac{\sqrt{3}}{2}(1+\sqrt{3}), 3\sqrt{2}/2, \frac{\sqrt{3}}{2}(1+\sqrt{3})$ respectively; but, the respective values of $d_{12}(\theta)$ are $\sqrt{3},\sqrt{6}, \sqrt{3}$. This shows that there are no equilateral triangle with respect to the Euclidean metric that is also equilateral with respect to the taxicab metric in $\mathbb{R}^2$. Note that the three triangles found are indeed \textit{isosceles} triangles with respect to the taxicab norm. In fact, a gereral result (see \cite{matschke}) shows that given a circle with two distance functions on it, $d_1$ and $d_2$, where $d_1$ is the restricted distance coming from a smooth embedding of $S^1$ into a Riemannian manifold, and $d_2$ is symmetric and continuous, then there are three points in $S^1$ forming a $d_1$-equilateral triangle and, at the same time, a ``$d_2$-isosceles'' triangle.  
\end{example}

\section{Existence of equilateral $p$-gons on deformed spheres}

Theorem \ref{equilateral} shows the existence of equilateral $p$-gons in $(\mathbb R^d,\rho)$ for $p$ prime if $p<d$; it does not apply to $m$-gons when $m$ is not prime, in particular it does not apply to quadrilaterals. For proving the existence of quadrilaterals, Fadell-Husseini index with integer coefficients is needed. In fact, we have the following result of Blagojevi\'{c} and Ziegler (see \cite{blagojevictetrahedra}):

\begin{theorem}\label{quadrilaterals}
For every injective continuous map $f:S^2\to \mathbb R^3$, there are four distinct points
 $y_0,y_1,y_2,y_3$ in the image of $f$ such that
\[d(y_0,y_1)=d(y_1,y_2)=d(y_2,y_3)=d(y_3,y_0)\]
and 
\[d(y_0,y_2)=d(y_1,y_3).\]
Here, $d$ is the Euclidean metric of $\mathbb R^3$. The result is still true if we change $d$ by any metric $\rho$ equivalent to $d$. 
\end{theorem}

From this theorem, the existence of equilateral quadrilaterals follows for any metric $\rho$ in $\mathbb R^3$ equivalent to the Euclidean metric (take $f$ as the inclusion $S^2\hookrightarrow \mathbb R^3$). The proof of this theorem relies on the use of Fadell-Husseini index with integer coefficients (It is shown in \cite{blagojevictetrahedra} that the use of mod. 2 coefficients is not enough).

We generalize this result of Blagojevi\'{c} and Ziegler in Theorem \ref{pgonsdeformed}. The problem is to determine when a pair $(d,p)$ is admissible. Consider a continuous and injective map $f:S^{d}\to\mathbb R^{d+1}$ and metrics $\rho_1,\rho_2,\ldots,\rho_{(p-1)/2}$ equivalent to the Euclidean metric on $\mathbb R^{d+1}$.
Consider the map $F=(f_1,f_2,\ldots,f_r):(S^{d})^p\to (\mathbb{R}^p)^{(p-1)/2}$ where 
$f_j(x_0,x_1,\ldots,x_{p-1})$ equals 
\begin{equation} (\rho_j(f(x_0), f(x_{1})), \rho_j(f(x_{j}), f(x_{2j})),\ldots,\rho_j(f(x_{(p-1)j}), f(x_0)),\end{equation}
for $(x_0,x_1,\ldots,x_{p-1})\in (S^d)^{p}$, $j=1,2,\ldots,(p-1)/2$. Let $L=\{(x,\ldots,x)\in (S^{d})^p:x\in S^{d}\}$ and $\Delta=\{(t,\ldots,t)\in \mathbb{R}^p: t\in\mathbb{R}\}$. Let us restrict $F$ to $X_{d,p}=(S^{d})^p\setminus L$. We assume that $im\ F\cap \Delta^{(p-1)/2}=\varnothing$ so that we can consider $F$ as a map $F:X_{d,p}\to Y_{p}:=(\mathbb{R}^p)^{(p-1)/2}\setminus\Delta^{(p-1)/2}$. The group $\mathbb{Z}/p$ acts on these spaces by permuting coordinates and making the map $F$ equivariant.

The orthogonal projection $Y_p\to \Delta^{(p-1)/2}\setminus \{0\}$ is $\mathbb Z/p$-equivariant, and so is the normalization map $\Delta^{(p-1)/2}\setminus \{0\}\to S(\Delta^{(p-1)/2})^{\perp}$, where $S(\Delta^{(p-1)/2})^{\perp}$ is the unit sphere of $(\Delta^{(p-1)/2})^{\perp}$. By composing, we get an equivariant map $Y_{p}\to S(\Delta^{(p-1)/2})^{\perp}\cong S^{(p-1)^2/1-1}$. 

\begin{proposition}\label{admissibleequivariant}
The pair $(d,p)$ is admissible if there is no $\mathbb{Z}/p$-equivariant map $X_{d,p}\to S(\Delta^{(p-1)/2})^{\perp}$.
\end{proposition}

\begin{proof}
If there is no $\mathbb{Z}/p$-equivariant map $X_{d,p}\to S(\Delta^{(p-1)/2})^{\perp}$, then for the map $F:X_{d,p}=(S^{d})^p\setminus L\to(\mathbb{R}^p)^{(p-1)/2}$ defined above we have $im\ F\cap \Delta^{(p-1)/2}\neq\varnothing$, that is, there exists $(x_0,x_1,\ldots,x_{p-1})\in X_{d,p}$ such that for every $j=1,2,\ldots,(p-1)/2$, 
\[\rho_j(f(x_0), f(x_{j}))= \rho_j(f(x_{j}), f(x_{2j}))=\cdots=\rho_j(f(x_{(p-1)j}), f(x_0)).\]
It only remains to show that $x_0, x_1,\ldots, x_{p-1}$ are distinct. In fact, if $x_{r}=x_{s}$ for some $r, s$, then there are $j\in\{1,2,\ldots, (p-1)/2\}$ and $t\in \mathbb{F}_p$ such that $r=tj$ and $s=(t+1)j$. Then, $\rho_j(f(x_{tj}),f(x_{(t+1)j}))=0$ and 
\[\rho_j(f(x_0), f(x_{j}))= \rho_j(f(x_{j}), f(x_{2j}))=\ldots=\rho_j(f(x_{(p-1)j}), f(x_0))=0,\]
which implies that $f(x_0)=f(x_1)=\cdots=f(x_{p-1})$, and since $f$ is injective, $x_0=x_1=\cdots=x_{p-1}$, which contradicts that $(x_0,x_1,\ldots,x_{p-1})\in X_{d,p}$. Thus, the points $x_0, x_1,\ldots, x_{p-1}\in S^d$ are distinct.
\end{proof}

We are going to consider the following problem, which is more general: to find conditions on a triplet $(d,r,p)$ such that there is no $\mathbb{Z}/p$-equivariant map $X_{d,p}\to S(\Delta^r)^{\perp}$.

\subsection{The index of the sphere $S(\Delta^{r})^{\perp}$}

The unit sphere $S(\Delta^{r})^{\perp}$ is a sphere of dimension $r(p-1)-1$ with a free action of $\mathbb{Z}/p$. Then, by Proposition \ref{indexzivaljevic}, $ii)$, we have 
\begin{equation} Ind_{\mathbb{Z}/p,\mathbb{F}_p}S(\Delta^{r})^{\perp}=\langle b^{r(p-1)/2}\rangle\subseteq \mathbb{F}_p[a,b]/\langle a^2\rangle=H^{\ast}(B\mathbb{Z}/p;\mathbb{F}_p).\end{equation}

We are not going to compute explicitly the index of $X_{d,p}$, but rather we are going to give conditions on $d,r,p$ such that the element $b^{r(p-1)/2}\in H^{\ast}(B\mathbb{Z}/p;\mathbb{F}_p)$ does not belong to the index of $X_{d,p}$. We do this by using the Serre spectral sequence of the fibration $X_{d,p}\to (X_{d,p})_{\mathbb{Z}/p}\to B\mathbb{Z}/p$ to show that the element $b^{r(p-1)/2}\in H^{\ast}(B\mathbb{Z}/p;\mathbb{F}_p)=E_2^{\ast,0}$ survives forever. This means that $b^{r(p-1)/2}$ is not in $Ind_{\mathbb{Z}/p,\mathbb{F}_p}X_{d,p}$.

The $E_2$-page of the Serre spectral sequence of the fibration $X_{d,p}\to (X_{d,p})_{\mathbb{Z}/p}\to B\mathbb{Z}/p$, $E_2^{s,t}=H^s(B\mathbb{Z}/p; H^t(X_{d,p};\mathbb{F}_p))$ has to be understood as the cohomology of the group $\mathbb{Z}/p$ with coefficients in the $\mathbb{Z}/p$-module $H^{\ast}(X_{d,p};\mathbb{F}_p)$.

\subsection{The cohomology $H^{\ast}(X_{d,p};\mathbb{F}_p)$ as a $\mathbb{Z}/p$-module}

We will use Lefschetz duality (see \cite{munkres}, theorem 70.2, page 415):

\begin{theorem}\label{lefschetz}
If $(X,A)$ is a compact triangulated relative homology $n$-manifold which is orientable, then there are isomorphisms 
\[H^k(X\setminus A;G)\cong H_{n-k}(X,A;G),\]
for all $G$.
\end{theorem}

We apply theorem \ref{lefschetz} to $((S^d)^p, L)$ and $G=\mathbb{Z}/p$ to have 
\begin{equation} H^{\ast}(X_{d,p};\mathbb{F}_p)=H^{\ast}((S^{d})^p\setminus L;\mathbb{F}_p)=H_{pd-\ast}((S^{d})^p, L;\mathbb{F}_p).\end{equation}
To compute the homology $H_{\ast}((S^{d})^p, L;\mathbb{F}_p)$ we are going to use the long exact sequence of the pair $((S^d)^p, L)$. 
Since $L\cong S^{d}$, $H_i(L;\mathbb{F}_p)=\mathbb{F}_p$ if $ i=0,d$ and $H_i(L;\mathbb{F}_p)=0$ otherwise.

An induction argument using the the K\"unneth isomorphism yields that the only nonzero homology groups $H_i((S^{d})^p;\mathbb{F}_p)$ occurs when $i=0,d,2d,\ldots,pd$; besides, $H_{jd}((S^{d})^p;\mathbb{F}_p)$ is the direct sum of all terms of the form $A_1\otimes A_2\otimes \cdots \otimes A_p$, where $j$ of the $A_l$'s are equal to $H_{d}(S^{d};\mathbb{Z}/p)$ and the remaining $A_l$'s are equal to $H_0(S^{d};\mathbb{Z}/p)$. If $N:=(\mathbb{F}_p)^{\oplus(p)}$ with the action of $\mathbb{Z}/p$ given by permuting the $p$ copies of $\mathbb{F}_p$ in the obvious way, then we have
\begin{equation} H_{i}((S^{d})^p;\mathbb{F}_p)\cong
\begin{cases}
\mathbb{F}_p, & i=0,pd,\\
N^{\oplus \frac{1}{p}{p\choose i/d}}, & i=d,2d,\ldots, (p-1)d,\\
0,& \text{otherwise.} 
\end{cases}\end{equation}

The long exact sequence in homology of the pair $((S^{d})^p, L)$ gives isomorphisms
\begin{equation} H_{jd}((S^{d})^p,L;\mathbb{F}_p)\cong H_{jd}((S^{d})^p;\mathbb{F}_p)\end{equation}
for $j=2,3,\ldots,p$, and an exact sequence 
\begin{equation} 0\rightarrow H_{d+1}((S^{d})^p,L)\rightarrow H_{d}(L)\overset{\iota_{\ast}}{\rightarrow}H_{d}((S^{d})^p)\rightarrow H_{d}((S^{d})^p,L)\rightarrow 0,\end{equation}
where $\iota_{\ast}$ is the map induced by the inclusion $L\hookrightarrow (S^{d})^p$. We have that $H_{i}((S^{d})^p,L;\mathbb{F}_p)=0$ in any other case. 
 If $x$ is the generator of $H_{d}(L;\mathbb{F}_p)$ and $x_i$ the generator of $H_{d}((S^{d})^p;\mathbb{F}_p)$ carried by the $i$-th copy of $S^{d}$ in $(S^{d})^p$, then we have that $\iota_{\ast}(x)=x_1+\cdots+x_p$. Thus, $\iota_{\ast}$ is injective, $H_{d+1}((S^{d})^p,L;\mathbb{F}_p)=0$ and
\begin{align*}
H_{d}((S^{d})^p,L;\mathbb{F}_p)& \cong H_{d}((S^{d})^p;\mathbb{F}_p)/im\ \iota_{\ast}\\
                                      &=\langle x_1,\ldots,x_p\rangle/\langle x_1+\cdots+x_p\rangle\\
                                      &=:M.\\
\end{align*}
Thus, we have that 
\begin{equation}
H_i((S^d)^p,L;\mathbb{F}_p)=
\begin{cases}
M,& i=d,\\
N^{\oplus \frac{1}{p}{p\choose i/d}},& i=2d,\ldots, (p-1)d,\\
\mathbb{Z}/p,& i=pd,\\
0,& \text{otherwise.} 
\end{cases}
\end{equation}
and by theorem \ref{lefschetz} and using that ${p\choose i/d}={p\choose p-i/d}$):
\begin{equation}
H^i(X_{d,p};\mathbb{F}_p)=
\begin{cases}
\mathbb{Z}/p,& i=0,\\
N^{\oplus \frac{1}{p}{p\choose i/d}},& i= d,\ldots, (p-2)d,\\
M,& t=(p-1)d,\\
0,& \text{otherwise.} 
\end{cases}
\end{equation}

\subsection{The Serre spectral sequence of the fibration $X_{d,p}\to (X_{d,p})_{\mathbb{Z}/p}\to B\mathbb{Z}/p$}

\noindent The $E_2$-page of the Serre spectral sequence of the fibration $X_{d,p}\to (X_{d,p})_{\mathbb{Z}/p}\to B\mathbb{Z}/p$ is given by:
\begin{equation}
E_2^{s,t}=H^s(\mathbb{Z}/p;H^t(X_{d,p};\mathbb{F}_p)).
\end{equation}
We have that $H^t(X_{d,p};\mathbb{F}_p)\neq 0$ only for $t=0,d,2d,\ldots,(p-1)d$. For $t=0$, $H^0(X_{d,p};\mathbb{F}_p)=\mathbb{F}_p$ so that $E_2^{\ast,0}=H^{\ast}(\mathbb{Z}/p;\mathbb{F}_p);\mathbb{F}_p[a,b]/\langle a^2\rangle$. To compute the cohomology groups $H^{s}(\mathbb{Z}/p;N)$ we consider the \textit{norm map} $\nu: N\to N$ given by multiplication by $1+\omega+\omega^2+\cdots +\omega^{p-1}$ (where $\mathbb Z/p=\langle \omega|\omega^p=1\rangle$) and have that for $s\geq 2$ even $H^{s}(\mathbb{Z}/p;N)=coker\ \overline{\nu}$ and for $s$ odd $H^{s}(\mathbb{Z}/p;N)=ker\ \overline \nu$, where $\overline \nu: N_{\mathbb Z/p}\to N^{\mathbb Z/P}$ is induced by $\nu$ (see \cite{brown}).This map $\overline \nu$ is an isomorphism, so we get the following: 
\begin{equation}
H^{s}(\mathbb{Z}/p;N)=
\begin{cases}
\mathbb{Z}/p, & s=0,\\
0,& s>0. 
\end{cases}
\end{equation}
For $t=d,2d,\ldots,(p-2)d$ we have 
\begin{equation}
E_2^{s,t}=H^s(\mathbb{Z}/p;N^{\oplus\frac{1}{p}{p\choose t/p}})=
\begin{cases}
(\mathbb{Z}/p)^{\frac{1}{p}{p\choose t/p}},& s=0,\\
0,& s>0.
\end{cases}
\end{equation}
We also have $E_2^{s,(p-1)d}=H^s(\mathbb{Z}/p;M)$. Otherwise, $E_2^{s,t}=0$. The $E_2$-term of the spectral sequence is shown in figure \ref{spectralsequence1}.
\begin{figure}
\begin{center}
\definecolor{gris}{rgb}{0.75,0.75,0.75}
\begin{tikzpicture}
\draw [color=gris,dash pattern=on 2pt off 2pt, xstep=1 cm,ystep=1 cm] (-.8,-.8) grid (9.3,7.3);
\draw[color=black] (-.8,0)--(6.2,0);
\draw[->,color=black] (6.8,0 ) -- (9.5,0 );

\draw[color=black] (0,-.8)--(0,3.2);
\draw[color=black] (0,3.8)--(0,6.2);
\draw[->,color=black] (0 ,6.8) -- (0 ,7.5);

\begin{scriptsize}
%\draw [fill=black] (2.0,1.5) circle (1.5pt);

\draw[color=black] (6.5,0) node {$\cdots$};

%%%%     EJE X    %%%%
\draw[color=black] (.5,-.5) node {$0$};
\draw[color=black] (1.5,-.5) node {$1$};
\draw[color=black] (2.5,-.5) node {$2$};
\draw[color=black] (3.5,-.5) node {$3$};
\draw[color=black] (4.5,-.5) node {$4$};
\draw[color=black] (5.5,-.5) node {$5$};
\draw[color=black] (6.5,-.5) node {$\cdots$};
\draw[color=black] (7.5,-.5) node {$\frac{r(p-1)}{2}$};
\draw[color=black] (8.5,-.5) node {$\cdots$};

%%%%    EJE Y    %%%%
\draw[color=black] (-.5,.5) node {$0$};
\draw[color=black] (-.5,1.5) node {$d$};
\draw[color=black] (-.5,2.5) node {$2d$};
\draw[color=black] (-.5,3.5) node {$\vdots$};
\draw[color=black] (-.5,4.5) node {$(\!p\!-\!2\!)d$};
\draw[color=black] (-.5,5.5) node {$(\!p\!-\!1\!)d$};
\draw[color=black] (-.5,6.5) node {$\vdots$};

%%%%%           FILA 0         %%%%%
\draw[color=black] (.5,.5) node {$ \langle 1\rangle$};
\draw[color=black] (1.5,.5) node {$ \langle a\rangle$};
\draw[color=black] (2.5,.5) node {$ \langle b\rangle$};
\draw[color=black] (3.5,.5) node {$ \langle ab\rangle$};
\draw[color=black] (4.5,.5) node {$ \langle b^2\rangle$};
\draw[color=black] (5.5,.5) node {$ \langle ab^2\rangle$};
\draw[color=black] (6.5,.5) node {$ \cdots $};
\draw[color=black] (7.5,.5) node {$ \langle b^{\frac{\!r\!(\!p\!-\!1\!)\!}{2}}\rangle$};
\draw[color=black] (8.5,.5) node {$ \cdots$};

%%%%%           FILA 1         %%%%%
\draw[color=black] (.5,1.5) node {$ N$};
\draw[color=black] (1.5,1.5) node {$ 0$};
\draw[color=black] (2.5,1.5) node {$ 0$};
\draw[color=black] (3.5,1.5) node {$ 0$};
\draw[color=black] (4.5,1.5) node {$ 0$};
\draw[color=black] (5.5,1.5) node {$ 0$};
\draw[color=black] (6.5,1.5) node {$ \cdots$};
\draw[color=black] (7.5,1.5) node {$ 0$};
\draw[color=black] (8.5,1.5) node {$ \cdots$};

%%%%%           FILA 2         %%%%%
\draw[color=black] (.5,2.5) node {$ N^{\frac{p-1}{2}}$};
\draw[color=black] (1.5,2.5) node {$ 0$};
\draw[color=black] (2.5,2.5) node {$ 0$};
\draw[color=black] (3.5,2.5) node {$ 0$};
\draw[color=black] (4.5,2.5) node {$ 0$};
\draw[color=black] (5.5,2.5) node {$ 0$};
\draw[color=black] (6.5,2.5) node {$ \cdots$};
\draw[color=black] (7.5,2.5) node {$0 $};
\draw[color=black] (8.5,2.5) node {$ \cdots$};

%%%%%           FILA 3         %%%%%
\draw[color=black] (.5,3.5) node {$ \vdots$};
\draw[color=black] (1.5,3.5) node {$\vdots $};
\draw[color=black] (2.5,3.5) node {$\vdots $};
\draw[color=black] (3.5,3.5) node {$\vdots $};
\draw[color=black] (4.5,3.5) node {$\vdots $};
\draw[color=black] (5.5,3.5) node {$\vdots $};
\draw[color=black] (6.5,3.5) node {$\cdots $};
\draw[color=black] (7.5,3.5) node {$ \vdots$};
\draw[color=black] (8.5,3.5) node {$ \cdots$};

%%%%%           FILA 4         %%%%%
\draw[color=black] (.5,4.5) node {$ N^{\frac{p-1}{2}}$};
\draw[color=black] (1.5,4.5) node {$ 0$};
\draw[color=black] (2.5,4.5) node {$ 0$};
\draw[color=black] (3.5,4.5) node {$ 0$};
\draw[color=black] (4.5,4.5) node {$ 0$};
\draw[color=black] (5.5,4.5) node {$ 0$};
\draw[color=black] (6.5,4.5) node {$ \cdots$};
\draw[color=black] (7.5,4.5) node {$ 0$};
\draw[color=black] (8.5,4.5) node {$ \cdots$};

%%%%%           FILA 5         %%%%%
\draw[color=black] (.5,5.5) node {$ E_{\!2\!}^{\!0\!,\!(\!p\!-\!1\!)\!d\!}$};
\draw[color=black] (1.5,5.5) node {$  E_{\!2\!}^{\!1\!,\!(\!p\!-\!1\!)\!d\!}$};
\draw[color=black] (2.5,5.5) node {$  E_{\!2\!}^{\!2\!,\!(\!p\!-\!1\!)\!d\!}$};
\draw[color=black] (3.5,5.5) node {$  E_{\!2\!}^{\!3\!,\!(\!p\!-\!1\!)\!d\!}$};
\draw[color=black] (4.5,5.5) node {$  E_{\!2\!}^{\!4\!,\!(\!p\!-\!1\!)\!d\!}$};
\draw[color=black] (5.5,5.5) node {$  E_{\!2\!}^{\!5\!,\!(\!p\!-\!1\!)\!d\!}$};
\draw[color=black] (6.5,5.5) node {$ \cdots$};
\draw[color=black] (7.5,5.5) node {$  E_{\!2\!}^{\!\frac{\!r\!(\!p\!-\!1\!)\!}{2}\!,\!(\!p\!-\!1\!)\!d\!}$};
\draw[color=black] (8.5,5.5) node {$ \cdots$};

%%%%%           FILA 6         %%%%%
\draw[color=black] (.5,6.5) node {$ \vdots$};
\draw[color=black] (1.5,6.5) node {$ \vdots$};
\draw[color=black] (2.5,6.5) node {$ \vdots$};
\draw[color=black] (3.5,6.5) node {$ \vdots$};
\draw[color=black] (4.5,6.5) node {$ \vdots$};
\draw[color=black] (5.5,6.5) node {$ \vdots$};
\draw[color=black] (6.5,6.5) node {$ \cdots$};
\draw[color=black] (7.5,6.5) node {$ \vdots$};
\draw[color=black] (8.5,6.5) node {$ \cdots$};

\end{scriptsize}

\end{tikzpicture}
\caption{$E_2$-page of the Serre spectral sequence}
\label{spectralsequence1}
\end{center}
\end{figure}

The element $b^{r(p-1)/2}$ lives in $E_2^{r(p-1),0}=H^{r(p-1)}(\mathbb{Z}/p;\mathbb{F}_p)$. With our description of $E_2^{\ast,\ast}$ we see that there are only two differentials that might hit elements in $E_2^{r(p-1),0}$, that are $d_{r(p-1)}:E_2^{0, r(p-1)-1}\to E_2^{r(p-1),0}$ and $d_{(p-1)d+1}:E_2^{(r-d)(p-1)-1, d(p-1)}\to E_2^{r(p-1),0}$.

\begin{proposition}
If the two differentials $d_{r(p-1)}:E_2^{0, r(p-1)-1}\to E_2^{r(p-1),0}$ and $d_{(p-1)d+1}:E_2^{(r-d)(p-1)-1, d(p-1)}\to E_2^{r(p-1),0}$ are zero, then $b^{r(p-1)/2}$ does not belong to $Ind_{\mathbb{Z}/p,\mathbb{F}_p}X_{d,p}$, that is, 
$Ind_{\mathbb{Z}/p,\mathbb{F}_p}S(\Delta^{r})^{\perp}\nsubseteq Ind_{\mathbb{Z}/p,\mathbb{F_p}}X_{d,p}.$
In particular, there is no $\mathbb{Z}/p$-equivariant map $X_{d,p}\to S(\Delta^r)^{\perp}$.
\end{proposition}

If $r\leq d$, then $E_2^{(2r-d)(p-1)-1, d(p-1)}=0$. On the other hand, $E_2^{0, r(p-1)-1}\neq 0$ implies that $r(p-1)-1=jd$ for some $j=1,2,\ldots, p-1$. The result is that if $r\leq d$ and $r(p-1)$ is not of the form $jd+1$ for any $j=1,2,\ldots,p-1$, then the two differentials above are zero. This implies that $b^{r(p-1)/2}$ is not in the index $Ind_{\mathbb{Z}/p,\mathbb{Z}/p}X_{d,p}$. Thus we have:

\begin{proposition}\label{triplets}
If $r\leq d$ and $r(p-1)$ is not of the form $jd+1$ for any $j=1,2,\ldots,p-1$, then there is no $\mathbb{Z}/p$-equivariant map $X_{d,p}\to S(\Delta^r)^{\perp}$.
\end{proposition}

Theorem \ref{pgonsdeformed} follows from proposition \ref{admissibleequivariant} and proposition \ref{triplets} by taking $r=(p-1)/2$. A direct consequence of Theorem \ref{pgonsdeformed} is:

\begin{corollary}\label{corpgons}
If $p\leq 2d+1$ and $d$ is even, then $(d,p)$ is admissible.
\end{corollary}

\begin{example}
Corollary \ref{corpgons} implies that $(2,5)$ is an admissible pair. In particular, taking the Euclidean metric $d$ of $\mathbb{R}^3$ and a continuous injective map $f:S^2\to \mathbb{R}^3$, there are 5 distinct points $y_0,y_1,y_2,y_3,y_4$ in the image of $f$ such that 
\[d(y_0,y_1)=d(y_1,y_2)=d(y_2,y_3)=d(y_3,y_4)=d(y_4,y_0)\]
and 
\[d(y_0,y_2)=d(y_2,y_4)=d(y_4,y_1)=d(y_1,y_3)=d(y_3,y_0).\]
\end{example}

\begin{example}
The triplet $(3,5)$ is admissible. In fact we have that $(p-1)^2/2=8$ and for $j=1,2,3,4$; $jd+1$ gives the values $4,7,10,13$.\\
Thus, for any continuous and injective map $f:S^3\to \mathbb{R}^4$, two metrics $\rho_1$ and $\rho_2$ that induce the usual topology of $\mathbb{R}^{4}$, there are 5 distinct points in the image of $f$, $y_0=f(x_0), y_1=f(x_1), y_2=f(x_2), y_3=f(x_3), y_4=f(x_4)$, such that
\[\rho_1(y_0, y_1) =\rho_1(y_1,y_2)=\rho_1(y_2,y_3)=\rho_1(y_3,y_4)=\rho_1(y_4,y_0),\]
\[\rho_2(y_0, y_2) =\rho_2(y_2,y_4)=\rho_2(y_4,y_1)=\rho_2(y_1,y_3)=\rho_2(y_3,y_1).\]

\end{example}

\end{document}